\documentclass[11pt]{article}
 \usepackage[T1]{fontenc}
\usepackage{amsthm}
\usepackage{amssymb}
\usepackage{amsmath}
\usepackage{comment}
\usepackage{thm-restate}
\usepackage{latexsym}
\usepackage{graphicx,enumerate,color}
\usepackage{url} 
\usepackage[affil-it]{authblk}
\usepackage{hyperref}
\usepackage[noabbrev,capitalise]{cleveref}

\usepackage{color}
\usepackage[normalem]{ulem}

\usepackage[margin=1in]{geometry}

\newcommand{\eps}{\varepsilon}

\theoremstyle{plain}
\newtheorem{thm}{Theorem}[section]
\newtheorem{lem}[thm]{Lemma}

\newtheorem{conj}[thm]{Conjecture}

{\noindent \emph{Proof.} {}{#1}{}}{\hfill
	$\Diamond$\vspace{1em}}

\theoremstyle{plain} 
\newcommand{\thistheoremname}{}
\newtheorem{genericthm}[section]{\thistheoremname}

\theoremstyle{definition}

\newcounter{counter}
\newcommand{\less}{\setminus}
 
 \def\dfn#1{{\sl #1}}

\title{Minimizing the number of edges in  $(C_4, K_{1,k})$-co-critical graphs}
\author{Gang Chen$^{a,}$\thanks{Supported by NSFC under the grant numbers 12161066 and 12061056  and Ningxia Natural Science Foundation under grant numbers 2021AAC03016 and 2023AAC020001. E-mail address: {\tt  chen\_g@nxu.edu.cn.}}\hskip 1cm Chenchen Ren$^{a,}$\thanks{E-mail address: {\tt  ccrjyb528@163.com}.}\hskip 1cm   Zi-Xia Song$^{b,}$\thanks{Supported by  NSF grant    DMS-2153945. E-mail address: {\tt Zixia.Song@ucf.edu}.}}

  \affil{
  { \small $^a${School of Mathematics and Statistics, Ningxia University, Yinchuan, Ningxia 750021, China}}
   \\{ \small $^b${Department  of Mathematics, University of Central Florida, Orlando, FL 32816, USA}}
     }

\date{}

\begin{document}
\maketitle
\begin{abstract}
Given graphs $H_1, H_2$, a \{red, blue\}-coloring of the edges of a graph $G$  is a critical coloring  if $G$ has neither  a red $H_1$ nor a blue $ H_2$. A non-complete graph $G$ is $(H_1, H_2)$-co-critical if  $G$ admits a critical coloring,  but $G+e$ has no critical coloring for    every edge $e$ in the complement of $G$.
 Motivated by a conjecture of Hanson and Toft from 1987, we    study  the minimum number of edges over all  $(C_4, K_{1,k})$-co-critical graphs on $n$ vertices. We show that for all $k \ge 2 $ and $ n \ge k +\lfloor \sqrt {k-1} \rfloor +2$, if $G$ is a $(C_4,K_{1,k})$-co-critical graph on $n$ vertices, then \[e(G) \ge \frac{(k+2)n}2-3- \frac{(k-1)(k+ \lfloor \sqrt {k-2}\rfloor)}2.\] Moreover, this linear bound is asymptotically best possible for all  $k\ge3$ and $n\ge3k+4$. It is worth noting that our constructions for the case when $ k$ is even   have at least three different  critical colorings. For $k=2$, we obtain the sharp bound for the minimum number of edges  of $(C_4, K_{1,2})$-co-critical graphs on $n\ge5 $ vertices by showing that all such graphs have at least $2n-3$ edges. Our proofs  rely on the structural properties of  $(C_4,K_{1,k})$-co-critical graphs and a result of Ollmann  on the minimum number of edges of $C_4$-saturated graphs.

\end{abstract}

\baselineskip 16pt

\section{Introduction}

All graphs considered in this paper are finite, and without loops or multiple edges. For a graph $G$, we will use $V(G)$ to denote the vertex set, $E(G)$ the edge set, $|G|$ the number of  vertices of $G$, $e(G)$ the number of edges of $G$, $N_G(x)$ the neighborhood of vertex $x$ in $G$, $\delta(G)$ the minimum degree, $\Delta(G)$ the maximum degree, $\alpha(G)$ the independence number, and $\overline{G}$ the complement of $G$. If $A, B \subseteq V(G)$ are disjoint, we say that $A$ is \dfn{complete to} $B$ if every vertex in $A$ is adjacent to every vertex in $B$; and $A$ is \dfn{anti-complete to} $B$ if no vertex in $A$ is adjacent to a vertex in $B$. The subgraph of $G$ induced by $A$, denoted $G[A]$, is the graph with vertex set $A$ and edge set $\{xy \in E(G): x, y \in A\}$. We denote by $B \less A$  the set $B - A$,  and $G \less A$ the subgraph obtained from $G$ by deleting all vertices in $A$, respectively.   For any edge $e$ in $\overline{G}$, we use $G+e$ to denote the graph obtained from $G$ by adding the new edge $e$. The {\dfn{join}} $G+H$ (resp.~{\dfn{union}} $G\cup H$) of two vertex disjoint graphs $G$ and $H$ is the graph having vertex set $V(G) \cup V(H)$  and edge set $E(G) \cup E(H) \cup \{xy:   x\in V(G),  y\in V(H)\}$ (resp. $E(G)\cup E(H)$). Given two isomorphic graphs G and H, we may (with a slight but common abuse of notation) write $G = H$.   We use $K_n$,  $C_n$, $P_n$ and $T_n$ to denote the complete graph,    cycle, path and a tree on  $n$ vertices, respectively; and $K_{m,n}$ to denote the complete bipartite graph with partite sets of order $m,n$.  For any positive integer $k$, we write $[k]$ for the set $\{1, 2, \dots, k\}$. We use the convention $``A:="$ to mean that $A$ is defined to be the right-hand side of the relation.\medskip

Given an integer $k \ge 1 $ and graphs $G$, ${H}_1, \ldots, {H}_k$, we write \dfn{$G \rightarrow ({H}_1, \ldots, {H}_k)$} if every $k$-coloring of $E(G)$ contains a monochromatic  ${H}_i$ in color $i$ for some $i\in[k]$. The classical \dfn{Ramsey number}  $r({H}_1, \dots, {H}_k)$ is the minimum positive integer $n$ such that $K_n \rightarrow ({H}_1, \dots, {H}_k)$. Following~\cite{Galluccio1992,Nesetril1986}, a non-complete graph $G$ is $(H_1, \dots, H_k)$-\dfn{co-critical} if $G \not\rightarrow ({H}_1,\dots, {H}_k)$, but $G + e \rightarrow ({H}_1,\dots, {H}_k)$ for every edge e in $\overline{G}$. The notion of co-critical graphs was initiated by Ne$\check{s}$et$\check{r}$il~\cite{Nesetril1986} in 1986.  We refer the reader to a recent paper by Zhang and and the third author~\cite{SongZhang21} for further background on $(H_1, \dots, H_k)$-co-critical graphs. Hanson and Toft~\cite{Hanson1987} from 1987 observed that for all $n \ge r:=r(K_{t_1},\dots,K_{t_k})$, the graph $K_{r-2}+ \overline K_{n-r+2}$ is $(K_{t_1},\dots, K_{t_k})$-co-critical with $(r - 2)(n - r + 2) + \binom{r - 2}{2}$ edges. They further   conjectured that no $(K_{t_1},\dots, K_{t_k})$-co-critical graph on $n$ vertices can have fewer than $e(K_{r-2}+\overline K_{n-r+2})$ edges.

\begin{conj}[Hanson and Toft~\cite{Hanson1987}]\label{HTC}  
Let $r = r(K_{t_1}, \dots, K_{t_k})$. Then every $(K_{t_1},\dots, K_{t_k})$-co-critical graph on $n$ vertices has at least
\begin{align*}
(r - 2)(n - r + 2) + \binom{r - 2}{2}
\end{align*}
edges. This bound is best possible for every $n$.
\end{conj}
\medskip

Conjecture~\ref{HTC} remains wide open.  It was shown in \cite{Chen2011} that every $(K_3,K_3)$-co-critical graph on $n\ge 56$ vertices has at least $4n-10$ edges. This settles the first non-trivial case of Conjecture~\ref{HTC} for sufficiently large $n$.    
Inspired by Conjecture~\ref{HTC}, Rolek and the third author~\cite{RolekSong18}   initiated the study of the minimum number of possible edges over all $(K_t, \mathcal{T}_k)$-co-critical graphs, where $\mathcal{T}_k$ denotes the family of all trees on $k$ vertices.  They  proved that for all $k\ge 4$ and $n$ sufficiently large, every $(K_3, \mathcal{T}_k)$-co-critical graph   on $n $ vertices has at least $(\frac{3}{2}+\frac{1}{2}\left\lceil \frac{k}{2}\right\rceil)n- (\frac{1}{2} \lceil\frac{k}{2}\rceil + \frac{3}{2})k-2$  edges.  This bound is asymptotically best possible for all such $k,n$. 
 Shortly after, Zhang and the third author~\cite{SongZhang21}   proved that for all $t\ge4$ and $k\ge  \max\{6,t\}$, there exists a constant $\ell(t,k)$ such that, for all   $n\ge (t-1)(k-1)+1$, if $G$ is a $(K_t, \mathcal{T}_k)$-co-critical graph on $n$ vertices, then
\begin{align*}
e(G)\ge \left(\frac{4t-9}{2}+\frac{1}{2}\left\lceil \frac{k}{2}\right\rceil\right)n-\ell(t,k).
\end{align*}
This bound is asymptotically best possible when $t\in \{4,5\}$ and all $k\ge 6$ and $n\ge(2t-3)(k-1)+\lceil {k/2} \rceil\lceil {k/2}\rceil-1 $.
 Chen, Ferrara, Gould, Magnant and Schmitt~\cite{Chen2011} in 2011 proved  that every  $(K_3,K_{1,2})$-co-critical graph on $n\ge 11$ vertices has at least $\left\lfloor{5n/2}\right\rfloor-5$ edges. This bound is sharp for all $n\ge 11$.
Very recently, Davenport,   Yang and the third author~\cite{DSY22} continued the work  to determine  the minimum number of possible edges over all $(K_t,K_{1,k})$-co-critical graphs on $n$ vertices.  

\begin{thm}[Davenport, Song and Yang~\cite{DSY22}]\label{ktstar}
For all $t\ge3$ and $k\ge3$, there exists a constant $\ell(t, k)$ such that  for all $n\in \mathbb{N}$ with  $n \ge (t-1)k+1$, if $G$ is a 
  $(K_t,K_{1,k})$-co-critical graph on  $n $ vertices, then    \[e(G)\geq \left(2t-4+\frac{k-1}{2}\right)n-\ell(t,k).\]
   This bound is asymptotically best possible    when $t\in\{3,4,5\}$ and all $k\ge3$ and $n \ge (2t-2)k+1$. 
 \end{thm}

\begin{thm}[Davenport, Song and Yang~\cite{DSY22}]\label{t:threeclaw}
 Every $(K_3, K_{1,3})$-co-critical graph on $n\ge13$ vertices  has at least   $  3n-4$ edges. This bound is sharp for all $n\ge13$.
\end{thm}

For  multicolor $k$, Ferrara, Kim and Yeager~\cite{Ferrara2014}     determined the minimum number of edges over all  $(H_1,\dots,H_k)$-co-critical graphs on $n$ vertices   when each $H_i$ is a matching.  Very recently, Chen, Miao, Zhang and the third author~\cite{P3} determined the minimum number of edges over all  $(H_1,\dots,H_k)$-co-critical graphs on $n$ vertices when each $H_i=P_3$.

In this paper, we  study the minimum number of possible edges over all $(C_4,K_{1,k})$-co-critical graphs on $n$ vertices. By a classic result of Parsons~\cite{parsons1975}, $r(C_4,K_{1,k})\le k +\lfloor \sqrt {k-1} \rfloor +2$ for all $k\ge2$.  Note that  every $(C_4,K_{1,k})$-co-critical graph has at least $r(C_4,K_{1,k})$ vertices. We prove  the following main results.  

\begin{thm}\label{C4K1k}
For all $k\ge 2$ and  $n\in\mathbb{N}$ with $ n \ge k +\lfloor \sqrt {k-1} \rfloor +2$, if $G$ is a $(C_4, K_{1,k})$-co-critical graph on $n$ vertices, then
\begin{align*}
 e(G) \ge \frac{(k+2)n}2-3- \frac{(k-1)(k+ \lfloor \sqrt {k-2}\rfloor)}2.
\end{align*}
Moreover, this linear bound is asymptotically best possible for all $k\ge3$ and $n\ge3k+4$.
\end{thm}

For $k=2$, we   obtain the sharp bound for the minimum number of edges of $(C_4, K_{1,2})$-co-critical graphs on $n\ge5 $ vertices.

\begin{restatable}{thm}{threeclaw}\label{C4P31}
Every $(C_4,K_{1,2})$-co-critical graph on $n\ge 5$ vertices has at least $2n-3$ edges. This bound is sharp for all $n\ge 5$.
\end{restatable}

Our proofs rely heavily on  a  result of Ollmann~\cite{Ollmann} on $C_4$-saturated graphs (see Theorem~\ref{C4saturated}).  We believe our  method  can be used to determine  the minimum number of edges of $(C_5, K_{1,k})$-co-critical graphs on $n$ vertices. 
The remainder of this paper is organized as follows.  We first establish some important structural properties of $(C_4, K_{1,k})$-co-critical graphs  in Section~\ref{sec:properties}. We then prove Theorem~\ref{C4K1k} in Section~\ref{sec:C4K1k} and Theorem~\ref{C4P31} in Section~\ref{sec:C4P3}.

\section{Structural properties of $(C_4, K_{1,k})$-co-critical graphs}\label{sec:properties}
 We need to  introduce more notation. Given a graph $H$, a graph $G$ is \emph{$H$-free} if $G$ does not contain $H$ as a subgraph;  and $G$ is \emph{$H$-saturated} if $G$ is $H$-free but $G+e$ is not $H$-free for every $e\in \overline{G}$.
For a $ (C_4, K_{1,k})$-co-critical graph $G$,  let   $\tau : E(G) \to \{\text{red, blue} \}$  be a  $2$-coloring of $E(G)$  with color classes  $E_r$ and $E_b$.  We use $G_{r}$ and $G_{b}$ to denote the spanning subgraphs of $G$  with edge sets  $E_r$ and $E_b$, respectively.
We define  $\tau$ to be a  \emph{critical coloring} of $G$ if $G$ has neither  a red $C_4$ nor a blue $ K_{1,k}$ under $\tau$, that is,  if  $G_r$ is $C_4$-free and $G_b$ is $K_{1,k}$-free.  For each vertex $v$ in $G$,   we use $d_r(v)$ and $N_r(v)$ to denote the degree and neighborhood of $v$ in $G_r$, respectively. Similarly, we define $d_b(v)$ and $N_b(v)$ to be the degree and neighborhood of $v$ in $G_b$, respectively.    

Given  a $ (C_4, K_{1,k})$-co-critical graph $G$, it is worth noting  that   $G$ admits at least one critical coloring but,  for any edge $e\in E(\overline{G})$,
$G+e$ admits  no critical coloring.  In order to estimate $e(G)$, we let $\tau : E(G) \to \{\text{red, blue} \}$  be a  critical coloring of $G$ such that $|E_r|$ is maximum among all critical colorings of $G$. Then $G_r$ is  $C_4$-free but, $G_r+e$ has a copy of $C_4$ for each $e\in E(\overline{G})$, that is, $G_r$ is $C_4$-saturated.
Our proofs rely heavily on  a  result of Ollmann~\cite{Ollmann} on $C_4$-saturated graphs.

\begin{thm}[Ollmann~\cite{Ollmann}]\label{C4saturated}
	Every $C_4$-saturated graph   on $n\ge5$ vertices has at least $\lfloor(3n-5)/2\rfloor$ edges. This bound is sharp for all $n\ge 5$. Moreover, every  $C_4$-saturated  with $n\ge5$ vertices and $\lfloor(3n-5)/2\rfloor$  edges  is isomorphic to the graph in  Figure~\ref{C4}(a)  if $n$ is even, and  is isomorphic to one of the graphs in  Figure~\ref{C4}(b,c)   if $n$ is odd.
\end{thm}

\begin{figure}[htbp]
	\centering
	\includegraphics[scale=0.5]{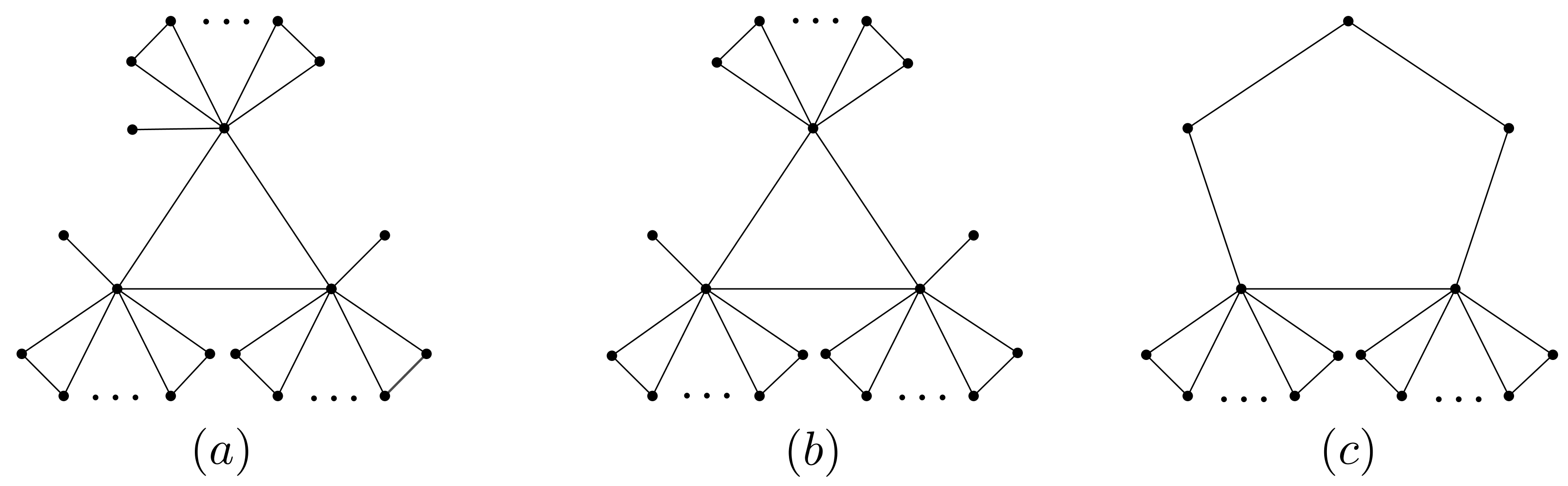}
	\caption{$C_4$-saturated graphs with $n\ge5$ vertices and $\lfloor(3n-5)/2\rfloor$ edges.}
	\label{C4}
\end{figure}

 We are now ready to prove some structural properties of $(C_4,K_{1,k})$-co-critical graphs.

\begin{lem}\label{lem:blue}
For all   $k\ge 2$, let $G$ be a $(C_4,K_{1,k})$-co-critical graph on $n\ge k +\lfloor \sqrt {k-1} \rfloor +2$ vertices. Let $\tau : E(G) \to\{red, blue\}$ be a critical coloring of $G$ with color classes $E_r$ and $E_b$. Let $S\subseteq V(G)$ be such that each vertex of $S$ has degree at most $k-2$ in $G_b$. Let $uv\in E(G)$. Then the following hold.

\begin{enumerate}[(i)]

\item  $S$ is a clique in $G$, that is, $G[S]=K_{|S|}$. Moreover, $\alpha(G_b[S])\le 3$ and $|S|\le k+\lfloor \sqrt {k-2} \rfloor$.

\item  If $|N(u)\cap N(v)|\ge 2k-2$, then $uv\in E_r$ and $|N_r(u)\cap N_r(v)|\le 1$.

\item  $|N(u)\cap N(v)|\le 2k-1$. Moreover, if $|N(u)\cap N(v)|=2k-1$, then $|N_r(u)\cap N_r(v)|=1$,   $|N_b(u)\cap (N(u)\cap N(v))|=|N_b(v)\cap (N(u)\cap N(v))|=k-1$, and   $N_b(u) \cap N_b(v)=\emptyset$.

 \end{enumerate}
\end{lem}

\begin{proof} To prove (i), suppose there exist $x, y\in S$ such that $xy\not\in E(G)$. Note that $d_b(x)\le k-2$ and $d_b(y)\le k-2$. But then we obtain a critical coloring of $G+xy$ from $\tau$ by coloring the edge $xy$ blue, contrary to the fact that $G$ is $(C_4,K_{1,k})$-co-critical. Thus  $S$ is a clique in $G$. Since $G_r$ is $C_4$-free, we see that $\alpha(G_b[S])\le 3$.  By the choice of $S$, $G_b[S]$ is $K_{1,k-1}$-free, and so $|S|\le r(C_4, K_{1, k-1})-1\le k+\lfloor \sqrt {k-2} \rfloor$, as desired.\medskip

To prove (ii), assume  $uv$ belongs to $2k-2$ triangles in $G$. Since $G_r$ is $C_4$-free, we see that $|N_r(u)\cap N_r(v)|\le 1$. Suppose $xy\in E_b$.  Then $d_b(u)+d_b(v)\ge (2k-2-1)+2$. It follows that $d_b(u)\ge k$ or $d_b(v)\ge k$. In either case, $G_b$ contains a copy of $K_{1,k}$, a contradiction. Thus $uv\in E_r$. \medskip

It remains to prove (iii). Suppose    $uv$ belongs to $2k-1$ triangles in $G$. By (ii), $uv\in E_r$ and $|N_r(u)\cap N_r(v)|\le 1$. Then $d_b(u)+d_b(v)\ge (2k-1)-1=2k-2$.  Since $G_b$ is $K_{1,k}$-free,  we see that  $uv$ belongs to exactly $2k-1$ triangles,  $|N_r(u)\cap N_r(v)|=1$,   $|N_b(u)\cap (N(u)\cap N(v))|=|N_b(v)\cap (N(u)\cap N(v))|=k-1$, and   $N_b(u) \cap N_b(v)=\emptyset$.  
\end{proof}

\section{Proof of Theorem~\ref{C4K1k}}\label{sec:C4K1k}

We are now ready to prove Theorem~\ref{C4K1k} which follows directly from Theorem~\ref{ksqrt} and Theorem~\ref{kngele}.

\begin{thm}\label{ksqrt}
For all $k\ge 2$ and $ n \ge k +\lfloor \sqrt {k-1} \rfloor +2$, if $G$ is a $(C_4,K_{1,k})$-co-critical graph on $n$ vertices, then
\begin{align*}
e(G)\ge\frac{(k+2)n}{2}-3-\frac{(k-1)(k+ \lfloor \sqrt {k-2}\rfloor )}{2}.
\end{align*}

\end{thm}

\begin{proof} Let $G$ be a $(C_4,K_{1,k})$-co-critical graph on $n\ge k +\lfloor \sqrt {k-1} \rfloor +2$ vertices. Then $G$ admits a critical coloring. Among all critical colorings of $G$, let $\tau:E(G)\to$ \{red, blue\} be a critical coloring of $G$ with color classes $E_r$ and $E_b$ such that $|E_r|$ is maximum. By the choice of $\tau$, $G_r$ is $C_4$-saturated and $G_b$ is $K_{1,k}$-free. Let $S:=\{v\in V(G)\mid d_b(v)\le k-2\}$. By Lemma~\ref{lem:blue}(i), $|S| \le k+\lfloor \sqrt {k-2} \rfloor $, and so
\begin{align*}
e(G_b) &\ge \left\lceil\frac{(k-1)(n-|S|)}{2}\right\rceil\\
&\ge \left\lceil\frac{(k-1)(n-k-\lfloor \sqrt {k-2} \rfloor)}{2}\right\rceil\\
&= \left\lceil\frac{(k-1)n}{2}-\frac{(k-1)(k+\lfloor \sqrt {k-2} \rfloor)}{2}\right\rceil.
\end{align*}
  By Theorem~\ref{C4saturated}, $e(G_r)\ge\lfloor(3n-5)/2\rfloor $. Consequently,

\begin{align*}
e(G)=e(G_r)+e(G_b)\ge\frac{(k+2)n}{2}-3-\frac{(k-1)(k+\lfloor \sqrt {k-2} \rfloor)}{2}, \end{align*}
as desired. This completes the proof of Theorem~\ref{ksqrt}.
\end{proof} \medskip

We next show that  the bound in Theorem~\ref{C4K1k}  is  asymptotically best possible  for all     $k \ge 3$ and $n\ge 3k + 4$.   \medskip

\begin{thm}\label{kngele}
For all $k\ge 3$ and $ n \ge 3k+3+\eps(2\eps^*-1)  $,  there exists  a $(C_4,K_{1,k})$-co-critical graph $G$ on $n$ vertices  such that
\begin{align*}
e(G)= \begin{cases}\frac{(k+2)n}{2}-\frac{(2+\varepsilon)k-4\varepsilon}{2} &\, \text{ if }\,     k \text{ is odd} \\
\frac{(k+2)n}{2}-\frac{(2-\varepsilon)k-2(1-2\varepsilon)}2 &\, \text{ if }\,     k \text{ is even} \\
\end{cases}
\end{align*}
where  $\varepsilon$ and  $\eps^*$ are  the remainders of $n$ and $k$ when divided by $2$, respectively.
\end{thm}

\begin{proof}  Let $ k, n,\eps , \eps^* $ be as given in the statement.  Note that $p:=n-2k-2-\eps(2\eps^*-1)$ is always even and $p\ge k+1$. Let  $R_{\varepsilon,\eps^*}$   be a  $k$-regular  graph on $p$ vertices  such that $E(R_{\varepsilon,\eps^*})$ can be decomposed into $k$ pairwise disjoint  perfect matchings, and let $H:= K_{{k-1},{k-1}}$ be the complete bipartite graph with partite sets $X$ and $Y$.    Finally,  let $x\in X$ and $y\in Y$ and define $H_{\eps^*}$ to be $H$ when $\eps^*=1$ and $H\less\{xy\}$ when $\eps^*=0$. We now construct a $(C_4,K_{1,k})$-co-critical graph on $n$ vertices which yields the desired number of edges.

 Let  $G_{\varepsilon,\eps^*} $ be the graph obtained from  disjoint copies of $H_{\eps^*}$ and  $R_{\varepsilon,\eps^*}$ by adding the $4+(2\eps^*-1)\varepsilon$ new vertices from $A:=\{x_1,\ldots, x_{4+(2\eps^*-1)\varepsilon}\}$ and then,   when $\eps^*=1$, joining $x_1$ to all vertices in $V(H_{\eps^*})\cup \{x_2,x_3,x_4 \}$,  $x_2$ to all vertices in $V(H_{\eps^*})\cup V(R_{\varepsilon,1})\cup (A\less\{x_1, x_2, x_3\})$,  and edges between each pair of vertices in $A\less\{x_1, x_2\}$;  and when $\eps^*=0$, joining $x_1$ to all vertices in $V(H_{\eps^*})\cup \{x_2,\ldots, x_{4-\eps} \}$,   $x_2$ to all vertices in $V(H_{\eps^*})\cup V(R_{\varepsilon,0})\cup \{x_{4-\varepsilon}\}$ and  $x$ to $x_{4-\varepsilon}$. Furthermore,    $x_3$ is adjacent to  both $x_4$  and   $y$ when $\eps=0$ and $\eps^*=0$.
 The construction of  $G_{\varepsilon, 1}$  is  depicted in Figure~\ref{C41} when $k=7$, and the construction of  $G_{\eps, 0}$  is  depicted in Figure~\ref{C4K1,81} when $k=8$ and $n$ is odd, and in Figure~\ref{C4K1,82} when $k=8$ and $n$ is even.
  It can be easily checked that  $e(G_{\varepsilon, 1}) =   ((k+2)n -(2+\varepsilon)k+4\varepsilon)/2$, and $e(G_{\varepsilon,0}) =  ((k+2)n -(2-\varepsilon)k+2(1-2\varepsilon))/2$.  It suffices to show that $G_{\varepsilon, \eps^*}$ is  $(C_4, K_{1,k})$-co-critical.

\begin{figure}[htbp]
\centering
\includegraphics[scale=0.7]{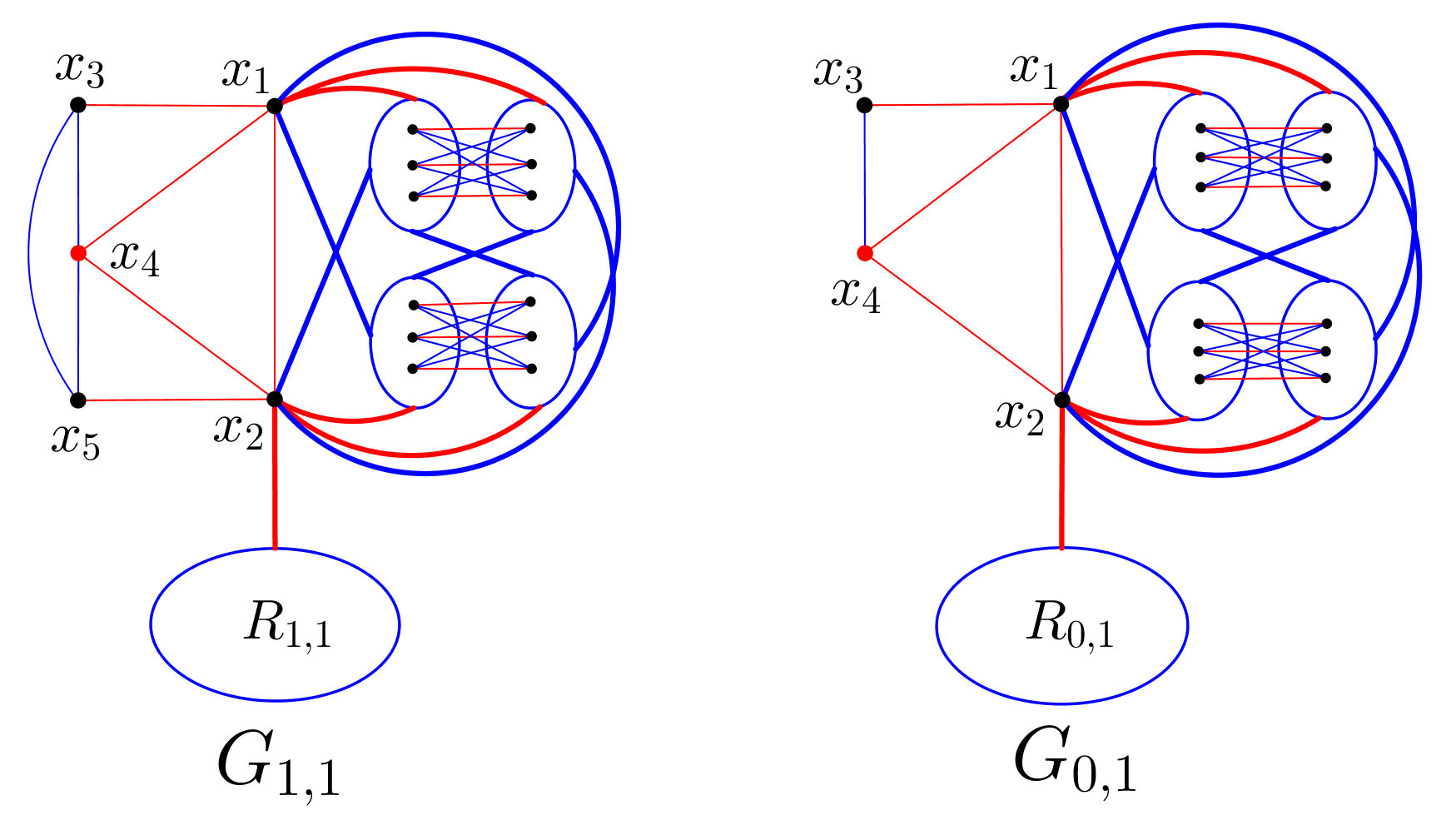}
\caption{Two $(C_4,K_{1,k})$-co-critical graphs with a unique critical coloring when $k=7$.}
\label{C41}
\end{figure}

 Let $\tau:E(G_{\varepsilon, \eps^*}) \rightarrow \{$red, blue$\}$ be an arbitrary critical coloring of   $G_{\varepsilon, \eps^*}$ with color classes $E_r$ and $E_b$.     Let $G^\tau_r$ and $G^\tau_b$ be  $G_r$ and $G_b$ under the coloring $\tau$, respectively.   Then $G^{\tau}_r$ is $C_4$-free and $G^{\tau}_b$ is $K_{1,k}$-free. Thus $\Delta(G^{\tau}_b)\le k-1$.
Note that $x_1x_2$ belongs to exactly $2k-1$ triangles in $G_{\varepsilon, \eps^*}$.     By Lemma~\ref{lem:blue}(ii,iii),  $x_1x_2\in E_r$, $|N_r(x_1)\cap N_r(x_2)|=1$, $N_b(x_i)\subseteq N(x_1)\cap N(x_2)$ for $i\in[2]$ and \medskip

\setcounter{counter}{0}
\noindent {\bf Claim\refstepcounter{counter}\label{unique}  \arabic{counter}.}
    $|N_b(x_1)\cap (N(x_1)\cap N(x_2)) |=|N_b(x_2)\cap (N(x_1)\cap N(x_2))|=k-1$, and  each vertex $v$ in $ N(x_1)\cap N(x_2)$ with $v\ne u$ belongs to exactly one of $N_b(x_1)$ and $N_b(x_2)$.  \medskip

\noindent {\bf Claim\refstepcounter{counter}\label{x1x3}  \arabic{counter}.}
    $ x_2$ is red-complete to $V(R_{\eps, \eps^*})$.    \medskip

\noindent {\bf Claim\refstepcounter{counter}\label{coloringR}  \arabic{counter}.} $E_r\cap E(R_{\varepsilon,\eps^*})$ is a perfect matching of $R_{\varepsilon,\eps^*}$ and all edges in $ E(R_{\varepsilon,\eps^*})\less E_r$ are colored blue.
     
 \begin{proof}
By Claim~\ref{x1x3}, $G_r^\tau[V(R_{\varepsilon,\eps^*})]$ is $P_3$-free because $G_r^\tau$ is $C_4$-free. Since $R_{\varepsilon,\eps^*}$ is $k$-regular and $\Delta(G_b^\tau)\le k-1$, we see that   $E_r\cap E(R_{\varepsilon,\eps^*})$ is a perfect matching of $R_{\varepsilon,\eps^*}$ and all edges in $ E(R_{\varepsilon,\eps^*})\less E_r$ are colored blue. 
\end{proof}

Let $u$ be the only vertex in $N_r(x_1)\cap N_r(x_2)$.    Then $ux_1, ux_2\in E_r$.  Since    $G_r^\tau$ is $C_4$-free, by Claim~\ref{unique}, we have  \medskip

\noindent {\bf Claim\refstepcounter{counter}\label{notu}  \arabic{counter}.}
  $v\notin  N_r(u)$  for each $v\in N(x_1)\cap N(x_2) $ with $v\ne u$. \medskip

\noindent {\bf Claim\refstepcounter{counter}\label{nored2}  \arabic{counter}.}
 For each $v\in X\cup Y$, we have $|N_r(v)\cap (X\cup Y) |\le 1$.
\begin{proof}
Suppose there exist $v\in X\cup Y$ and $y_1, y_2\in X\cup Y$ with $y_1\ne y_2$ such that $vy_1, vy_2\in E_r$.  We may assume that $v\in X$ and $y_1, y_2\in   Y$.  Note that $v, y_1, y_2\in N(x_1)\cap N(x_2)$. By Claim~\ref{unique},  each vertex in $\{v, y_1, y_2\}$ is adjacent to $x_1$ or $x_2$ in $G_r^\tau$, and so $G^{\tau}_r[\{x_1, x_2,  v, y_1, y_2\}]$ has a copy of  $C_4$, a contradiction.  
\end{proof}

\noindent {\bf Claim\refstepcounter{counter}\label{redU}  \arabic{counter}.}
  $ N_r(u)=\{x_1, x_2\} $.

\begin{proof}
Suppose $ N_r(u)\less \{x_1, x_2\}\ne \emptyset$. Let   $u'\in N_r(u)$ with  $u'\notin\{x_1, x_2\}$. Then $u'x_1, u'x_2\in E_b$ because $G^{\tau}_r$ is $C_4$-free. By Claim~\ref{notu}, $u'\notin N[x_1]\cap N[x_2]$. Thus  $u'\in A\less N[x_1]\cap N[x_2]$. It follows that     $\eps=\eps^*=0$, $u'=x_3$, and  $u\in\{x_4,y\}$.    Thus  $x_3=u'\in N_b(x_1)$, contrary to the fact that $N_b(x_1)\subseteq N(x_1)\cap N(x_2)$.  
\end{proof}

\noindent {\bf Claim\refstepcounter{counter}\label{vw}  \arabic{counter}.}
For each $vw\in E_r$ with $v, w\in N(x_1)\cap N(x_2)$,   $x_i$ is red-complete to $\{v,w\}$ and  $x_{3-i}$ is blue-complete to $\{v,w\}$ for some $i\in[2]$.

\begin{proof} This follows immediately from the fact that $G^\tau_r$ is $C_4$-free. 
\end{proof}

We first consider the case when $\eps^*=1$, that is, when $k$ is odd.  Let $\sigma: E(G_{\varepsilon, 1}) \rightarrow \{$red, blue$\}$ be defined as follows: $H_{\eps^*}=H$     has a perfect matching $M$ colored red;  $R_{\varepsilon, 1}$ has a perfect matching colored red;  all edges between $x_1$ and   $V(M_1)\cup \{x_2,x_3, x_4\}$ are colored red;  all edges between $x_2$ and   $V(M_2)\cup V(R_{\varepsilon, 1})\cup (A\less \{x_1,x_2, x_3\})$ are colored red; the remaining edges of $G_{\varepsilon,1}$ are all colored blue, where $\{M_1, M_2\}$ is an equal partition of $M$. The coloring    $\sigma $ is  depicted in   Figure~\ref{C41}  when $k=7$.   We next show that $\tau=\sigma$ up to symmetry.\medskip

Recall that $u$ is the only neighbor of $x_1$ and $x_2$ in $G_r^\tau$. We claim that $u=x_4$.  Suppose $u\ne x_4$. Then $u\in V(H)$. By symmetry, we may assume that $u\in X$.   Let  $v\in Y$.   By Claim~\ref{redU}, $ uv\in E_b$. By Claim~\ref{unique},  either $vx_1\in E_r$ or  $vx_2\in E_r$.     Since $d_b(v)\le k-1$, we see that  $N_r(v)\cap (X\less \{u\})\ne \emptyset$. By the arbitrary choice of $v$,  there exists $z\in X\less\{u\}$  such that $|N_r(z)\cap Y|\ge2$, contrary to Claim~\ref{nored2}.  This proves that $u=x_4$, as claimed. 
By Claim~\ref{unique},   $x_1x_3,  x_2x_{4+\varepsilon} \in E_r$, $G^\tau_b[A\less \{x_1, x_2\}]=K_{2+\varepsilon}$,     and $N_b(x_i) =N_r(x_{3-i})\cap V(H)$ for each $i\in[2]$. Let   $H^1:=G_{\eps,1}[N_r(x_1)\cap V(H)]$  and $H^2:=G_{\eps,1}[N_r(x_2)\cap V(H)]$. Then $|H^1|=|H^2|=k-1$ is even. 
 By Claim~\ref{coloringR},   $E_r\cap E(R_{\varepsilon,1})$ is a perfect matching of $R_{\varepsilon,1}$ and all edges in $ E(R_{\varepsilon,1})\less E_r$ are colored blue.   We next show that $E_r\cap E(H)$ is a perfect matching of $H$ and   $E_r\cap E(H^i)$ is a perfect matching of $H^i$ for each $i\in[2]$.  Let $v\in X$. We may assume that $v\in V(H^1)$. Since $d_b(v)\le k-1$ and $vx_2\in E_b$, we see that $vw\in E_r$ for some $w\in Y$.  Then $wx_1\in E_r$ by Claim~\ref{vw}. Thus   $N_r(v)\cap Y=\{w\}$ and $N_r(w)\cap X=\{v\}$ by Claim~\ref{nored2}.  This proves that   $E_r\cap E(H)$ is a perfect matching of $H$ and $E_r\cap E(H^i)$ is a perfect matching of $H^i$ for each $i\in[2]$. Thus all edges in $ E(H)\less E_r$ are colored blue. By symmetry, we may assume that $E_r\cap E(H)=M$ and $E_r\cap E(H^i)=M_i$ for each $i\in[2]$, and so $\tau=\sigma$ up to symmetry. \medskip

\begin{figure}[htbp]
\centering
\includegraphics[scale=0.6]{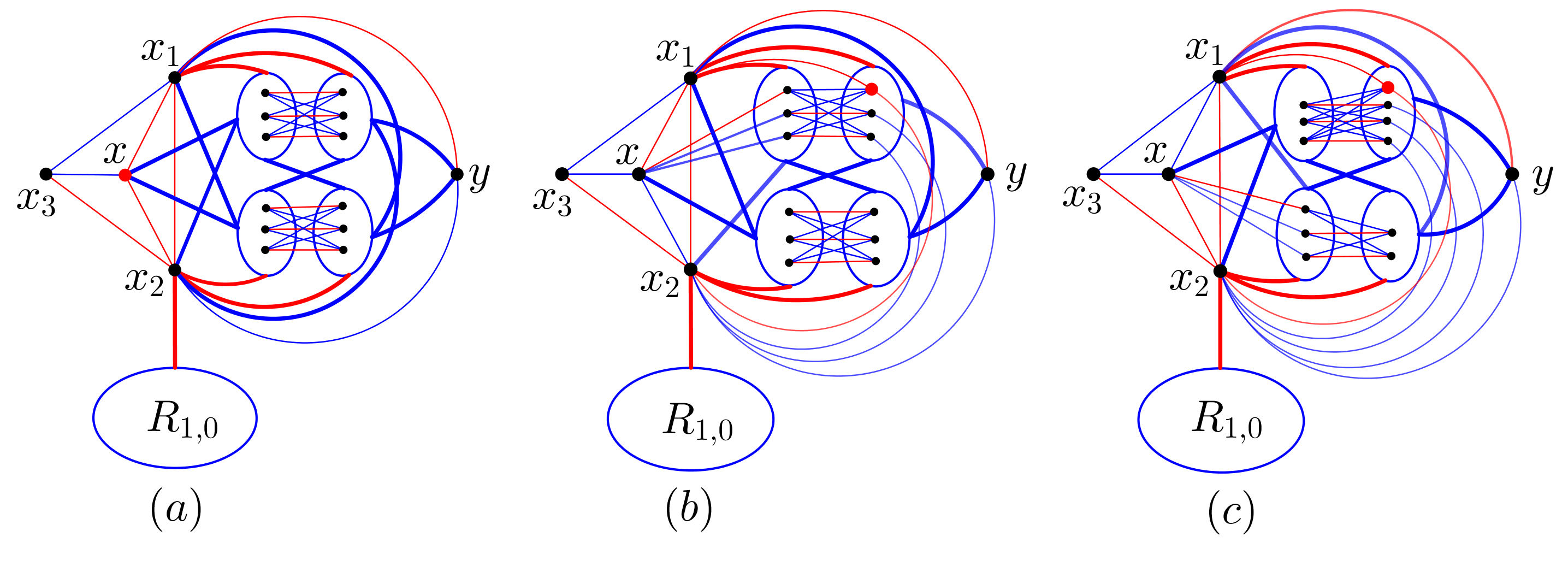}
\includegraphics[scale=0.6]{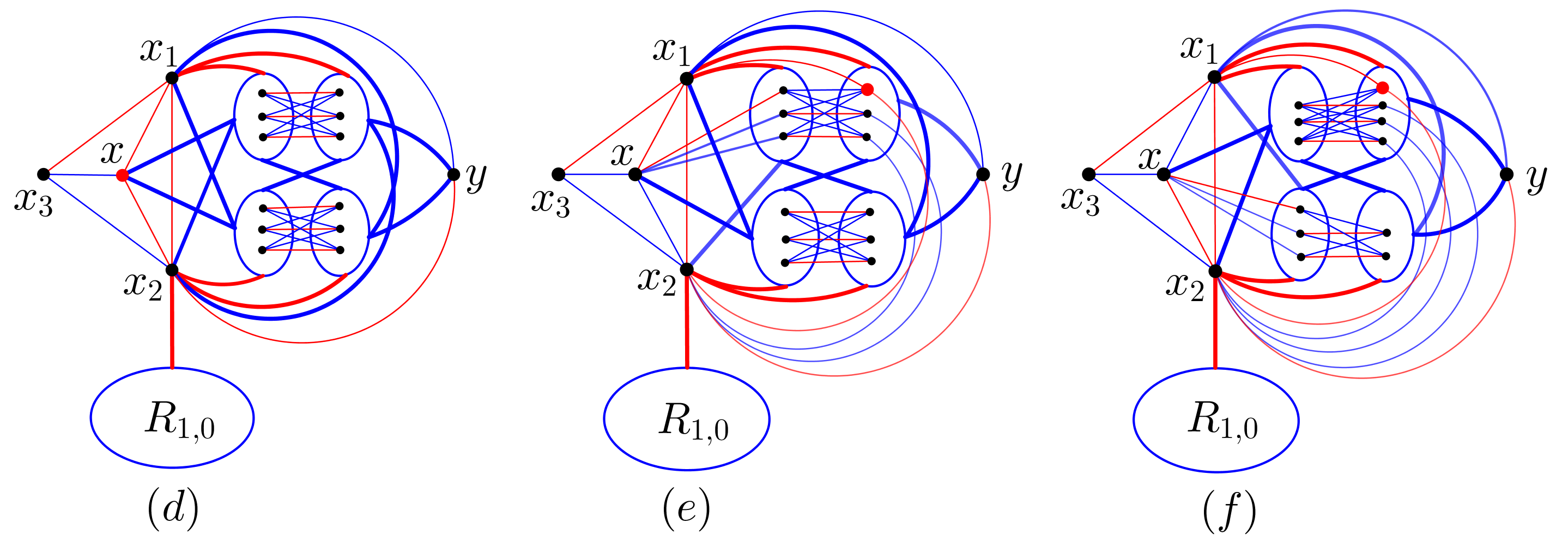}
\caption{A $(C_4,K_{1,k})$-co-critical graph with exactly six different critical colorings when $k=8$ and $n$ is odd, where the vertex $u$ is in red.}
\label{C4K1,81}
\end{figure}

\begin{figure}[htbp]
\centering
\includegraphics[scale=0.65]{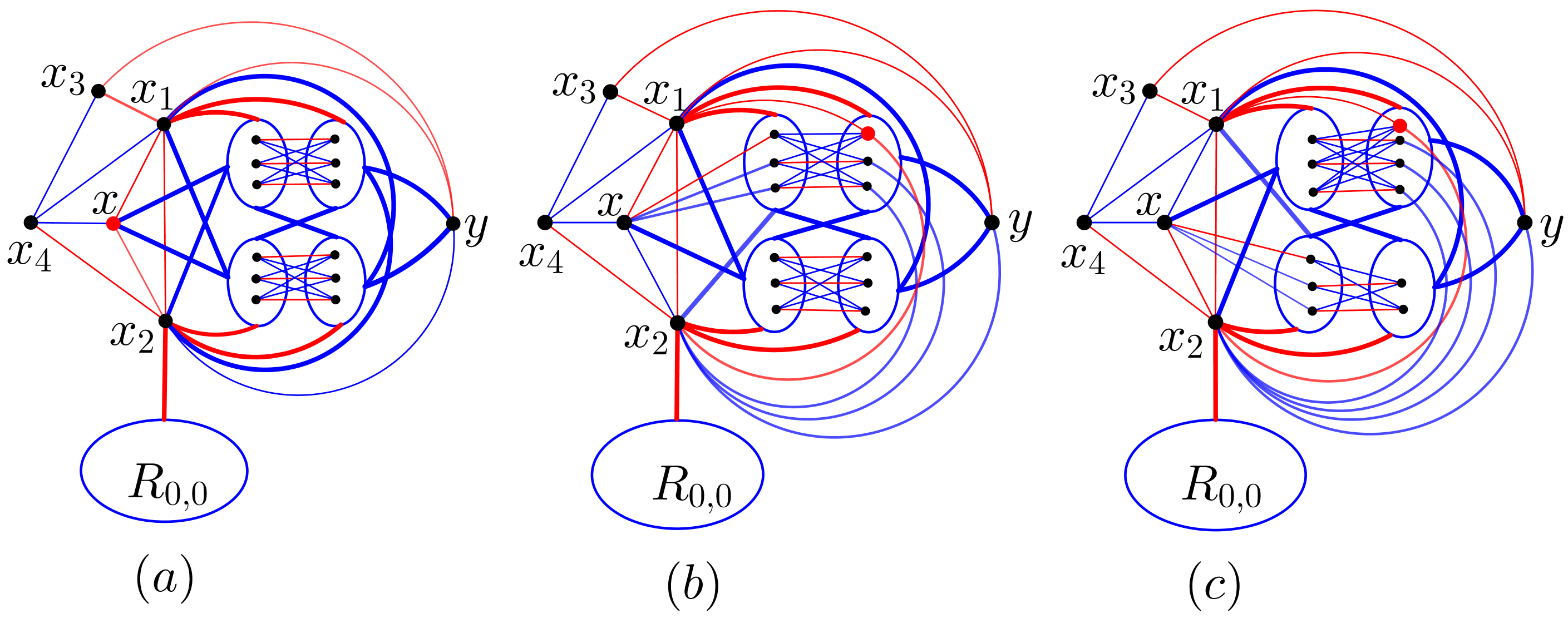}
\caption{A $(C_4,K_{1,k})$-co-critical graph with exactly three different critical colorings when $k=8$ and $n$ is even, where the vertex $u$ is in red.}
\label{C4K1,82}
\end{figure}

It remains to consider the case when $\eps^*=0$, that is, when $k$ is even.   Then $k\ge4$. This case is more involved as critical colorings of $G_{\eps, 0}$  are no longer  unique   up to symmetry. 
Recall that in this case $N(x_1)\cap N(x_2)=X\cup Y\cup \{x_{4-\eps}\}$, and $u$ is the only vertex in $N_r(x_1)\cap N_r(x_2)$.    
We next prove several claims. \medskip

\noindent {\bf Claim\refstepcounter{counter}\label{fullblue}  \arabic{counter}.}
 For each $v\in X\cup Y\cup V(R_{\eps,0})$,   we have $d_b(v)= k-1 $.
\begin{proof}
Suppose there exists $v\in X\cup Y\cup V(R_{\eps,0})$   such that $d_b(v)\le  k-2 $.  Note that $k\ge4$.  We first consider the case that  $x_3v\notin E(G)$.  If $d_b(x_3)\le k-2$, then  we obtain a critical coloring of $G_{\eps,0}+ vx_3$ from $\tau$ by coloring $vx_3$ blue,  a contradiction.  Thus $d_b(x_3)=k-1$, and so $k=4$ and $\eps=0$. In particular, $x_3\in N_b(x_1)$, contrary to the fact that $N_b(x_1)\subseteq N(x_1)\cap N(x_2)$. This proves that $x_3v\in E(G)$. Then $v=x$ when   $\eps=1$, and $v=y$  when $\eps=0$. In both cases, we have  $d_r(v)\ge3$ and so $u\ne v$ by Claim~\ref{redU}. By Claim~\ref{notu}, $v\notin N_r(u)$.  But then $|N_r(v)\cap Y|\ge2$ when $v=x$, contrary to Claim~\ref{nored2}. Thus $v=y$  and  $\eps=0$.  By Claim~\ref{nored2}, $yx_3\in E_r$.  
Note that $x_1x_3\in E_r$ by Claim~\ref{unique}. Thus $yx_2\in E_b$ because $G_r^\tau$ is $C_4$-free, and so $yx_1\in E_r$. Let $w\in X\cap N_r(y)$. But then $wx_1\in E_r$ or $x_2w\in E_r$,  and so $G_r^\tau[\{x_1, x_2, x_3, y, w\}]$ contains a copy of $C_4$, a contradiction.  
\end{proof}

\noindent {\bf Claim\refstepcounter{counter}\label{red1}  \arabic{counter}.}
 For each $v\in X\cup Y$ with $v\notin\{u, y\}$,   we have $|N_r(v)\cap (X\cup Y) |=1$.
\begin{proof} Suppose $|N_r(v)\cap (X\cup Y) |\ne1$ for some $v\in X\cup Y$ with $v \ne u$ and $v\ne y$. By Claim~\ref{nored2}, $|N_r(v)\cap (X\cup Y) |=0$. By Claim~\ref{unique}, $v$ belongs to exactly one of $N_b(x_1)$ and $N_b(x_2)$.  By 
Claim~\ref{fullblue}, $d_b(v)=k-1$, and so     $v= x$ and $xx_{4-\eps}\in E_r$. By Claim~\ref{vw}, 
$x_i$ is red-complete to $\{x, x_{4-\eps}\}$ for some $i\in[2]$.
 We may assume that $yx_\ell\in E_r$ for some $\ell\in[2]$.   For each $v\in N_b(x_\ell)\less\{x, x_{4-\eps}\}$, we see that  $|N_r(v)\cap (X\cup Y) |=1$. Note  that $G_r^\tau[N_b(x_\ell)]$ is $P_3$-free because $x_{3-\ell}$ is red-complete to $N_b(x_\ell)$; and  $|N_b(x_\ell)\less\{x, x_{4-\eps}\}|\in\{k-3, k-1\}$, in particular,  $|N_b(x_\ell)\less\{x, x_{4-\eps}\}|$ is odd. It follows that there exist $w\in N_b(x_\ell)\less\{x, x_{4-\eps}\}$ and $w'\in N_b(x_{3-\ell})$ such that $ww'\in E_r$, contrary to Claim~\ref{vw}.
\end{proof}

By Claim~\ref{unique}, 
    $|N_b(x_i) |=k-1$  and $N_b(x_i)\subseteq N_r(x_{3-i})\cap (N(x_1)\cap N(x_2)$ for each $i\in[2]$.    
Recall that   $ k-1\ge3$  is odd in this case. Let   $H^1:=G_{\eps,0}[N_b(x_1))]$  and $H^2:=G_{\eps,0}[N_b(x_2)]$. Then $|H^1|=|H^2|=k-1$ is odd and both $H^1$ and $H^2$ are $P_3$-free because $G^\tau_r$ is $C_4$-free. Claim~\ref{blue0} follows from Claim~\ref{vw} and Claim~\ref{red1}.   \medskip

\noindent {\bf Claim\refstepcounter{counter}\label{blue0}  \arabic{counter}.}
 For each $i\in[2]$, there exists $z_i\in N_b(x_i)$ such that $N_r(z_i)\cap (X\cup Y)=\emptyset$,   and so $E_r\cap E(H^i)$ is a perfect matching of $H^i\less \{z_i\}$. \medskip

 By Claim~\ref{red1}, $z_1, z_2\in\{x_{4-\eps}, y\}$.  We may assume $y=z_\ell$ for some $\ell\in [2]$. Then $y\in N_b(x_\ell)$ and $x_{4-\eps} =z_{3-\ell}\in N_b(x_{3-\ell})$. It follows from Claim~\ref{blue0} that  \medskip
 
\noindent {\bf Claim\refstepcounter{counter}\label{bluey}  \arabic{counter}.}
  $u\notin\{y, x_{4-\eps}\}$, $y$ is blue-complete to    $X\less\{x\}$, $xx_{4-\eps}\in E_b$ and $yx_{3-\ell}, x_{4-\eps}x_\ell\in E_r$.  Moreover, $yx_3\in E_r$ when $\eps=\eps^*=0$.\medskip
  
\noindent {\bf Claim\refstepcounter{counter}\label{uinX}  \arabic{counter}.}
$u\in X$.

\begin{proof} 
Suppose $u\notin X$.   By Claim~\ref{bluey},   $u\in Y\less\{y\}$ and $xx_{4-\eps}\in E_b$. For each $v\in X$, by    Claim~\ref{red1},    there exists a vertex $z\in Y$ with $z\notin\{u,y\}$ such that $vz\in E_r$.      Since $|X|>|Y\less\{u,y\}|$,    there must exist  $z'\in Y\less\{u,y\}$ such that $|N_r(z')\cap X|\ge2$,  contrary to Claim~\ref{red1}. \end{proof}

 By  Claim~\ref{blue0}, 
$E_r\cap E(H^\ell)$ is a perfect matching of  $H^\ell\less\{y\}$  and $E_r\cap E(H^{3-\ell})$ is a perfect matching of  $H^{3-\ell}\less\{x_{4-\eps}\}$. 
 Note that when $\eps=\eps^*=0$,  we have $yx_3\in E_r$ by Claim~\ref{bluey}, and so $yx_2\in E_b$, that is, $\ell=2$, because $G_r^\tau$ is $C_4$-free. By Claim~\ref{coloringR},   $E_r\cap E(R_{\varepsilon,0})$ is a perfect matching of $R_{\varepsilon,0}$ and all edges in $ E(R_{\varepsilon,0})\less E_r$ are colored blue.  By Claim~\ref{uinX}, we have $u\in X$. Assume first  
$u=x$.    Recall that $\ell\in [2]$,  and $yx_\ell\in E_b$ and $yx_{3-\ell}\in E_r$. The cases of $\tau$ are depicted below when $k=8$:   $\tau$ is as given in Figure~\ref{C4K1,81}(a,d) when  $\eps=1$ and $\eps^*=0$,  and  in Figure~\ref{C4K1,82}(a) when   $\eps= \eps^*=0$.  Assume next     $u\in X\less\{x\}$. By Claim~\ref{unique}, $x $  belongs to exactly one of $N_b(x_1)$ and $N_b(x_2)$. The cases of $\tau$ are depicted below when $k=8$: when $xx_1\in E_r$,  then $\tau$ is as given in Figure~\ref{C4K1,81}(b,e) when  $\eps=1$ and $\eps^*=0$,  and  in Figure~\ref{C4K1,82}(b) when   $\eps= \eps^*=0$. When $xx_2\in E_r$,  then $\tau$ is as given in Figure~\ref{C4K1,81}(c,f) when  $\eps=1$ and $\eps^*=0$,  and  in Figure~\ref{C4K1,82}(c) when   $\eps= \eps^*=0$.  \medskip

This completes the proof of Theorem~\ref{kngele}.
\end{proof}

\section{Proof of Theorem~\ref{C4P31}}\label{sec:C4P3}

In this section we prove  \cref{C4P31}, which we restate here for convenience.\threeclaw*

\begin{proof}
We first  prove  that every $(C_4,K_{1,2})$-co-critical graph on $n\ge5$ vertices has at least $2n-3$ edges. Suppose the statement is false. Let  $G$ be a  $(C_4,K_{1,2})$-co-critical graph on $n\geq 5$ vertices with $e(G)\leq 2n-4$.  Recall that  $G$ admits  at least one critical coloring. Among all critical colorings of $G$, let $\tau:E(G)\longrightarrow \{$red, blue$\}$ be a critical coloring of $G$ with color classes $E_r$ and $E_b$ such that $|E_r|$ is maximum.  By the choice of $\tau$, we see that $G_r$ is $C_4$-saturated and $G_b$ is $K_{1,2}$-free.  Thus $|E_r|=e(G_r)\ge \lfloor(3n-5)/2\rfloor$ by Theorem~\ref{C4saturated},  and  $\Delta(G_b)\leq 1$.  Let $S:=\{v\in V(G)\mid d_b(v)=0\}$. By Lemma~\ref{lem:blue}(i),  $|S|\le3$ and so $|E_b|=e(G_b)= (n-|S|)/2\ge(n-3)/2$.

  Suppose first   $n$ is odd.     Then $|S|=3$, else $|S|=1$ and  $e(G)=e(G_r)+e(G_b)\ge  (3n-5)/2 + (n-1)/2>2n-4$, a contradiction.  It follows that $e(G_r)=  (3n-5)/2$ and $e(G_b)=(n-3)/2$.  By Lemma~\ref{lem:blue}(i),  $ G_r[S]=K_3$.   By Theorem~\ref{C4saturated}, $G_r$ is   isomorphic to the graph given in Figure~\ref{C4}(b) or Figure~\ref{C4}(c).  Let $H$ be the  triangle in the center of Figure~\ref{C4}(b).  Note that every edge in $G_r[S]$ can be recolored blue without creating a blue $K_{1,2}$.  Let $S:=\{v_1, v_2, v_3\}$.  We may assume that $2\le d_r(v_1)\le d_r(v_2)\le d_r(v_3)$. 
 Suppose $N_r(v_2)=\{v_1, v_3\}$. Then $N_r(v_1)=\{v_2, v_3\}$,  $n\ge 7$ and $d_r(v_3)\ge 4$.  Let $u\in N_r(v_3)$ with $u\notin \{v_1, v_2\}$. Then $uv_1\notin E(G)$. Thus  $u$   belongs  to a triangle in $G_r$,  else  we obtain a critical coloring of $G+uv_1$ from $\tau$ by first coloring the edge $uv_1$ red and then recoloring the edge $v_2v_3$ blue, a contradiction.   It follows that  $G_r$ is   isomorphic to the graph given in Figure~\ref{C4}(b).  Then $v_3\in V(H)$.  We may further assume that $u\in V(H)$  such that there exists $ v\in N_r(u)$  with   $d_r(v)=1$. But then we obtain a critical coloring of $G+vv_1$ from $\tau$ by first coloring the edge $vv_1$ red and then recoloring the edge $v_1v_3$ blue, a contradiction. This proves that $ \{v_1, v_3\}\subsetneq N_r(v_2) $, and so $d_r(v_2)\ge3$. Thus
 $G_r$ is   isomorphic to the graph given in Figure~\ref{C4}(b) and $G_r[S]=H$. We may assume that  there exists $w\in N_r(v_3)$ such that  $d_r(w)=1$. But then we obtain a critical coloring of $G+wv_1$ from $\tau$ by first coloring the edge $wv_1$ red and then recoloring the edge $v_1v_2$ blue, a contradiction.

   Suppose next  $n$ is even. Since $e(G_b)=(n-|S|)/2$, we see that   $|S|  $ is even.  Then $|S|=2$, else $|S|=0$ and  $e(G)=e(G_r)+e(G_b)\ge  (3n-6)/2 + n/2>2n-4$, a contradiction.  It follows that $e(G_r)=  (3n-6)/2$ and $e(G_b)=(n-2)/2$.    By Theorem~\ref{C4saturated}, $G_r$ is   isomorphic to the graph given in Figure~\ref{C4}(a) and so $n\ge6$.   Let $H$ be the  triangle in the center of Figure~\ref{C4}(a) and let    $S:=\{v_1, v_2\}$.   By Lemma~\ref{lem:blue}(i),  $v_1v_2\in E_r$.  We may assume that $d_r(v_1)\le d_r(v_2) $.  Then $d_r(v_2)\ge 2$.   Suppose $d_r(v_2)=2$.  Then $d_r(v_1)=2$ and $v_1v_2$ belongs to a triangle in $G_r$.  Let $\{v_3\}=N_r(v_1)\cap N_r(v_2)$. Then  $v_3$ has a neighbor $u$ in $G_r$ such that $d_r(u)=1$.     But then  $uv_1\notin E(G)$ and    we obtain a critical coloring of $G+uv_1$ from $\tau$ by first coloring the edge $uv_1$ red and then recoloring the edge $v_1v_2$ blue, a contradiction. Thus $d_r(v_2)\ge 3$, and so $v_2 \in V(H)$.  Let $u\in V(H)$  with $u\ne v_2$ and let $v\in N_r(u)$ with $d_r(v)=1$.   If  $u=v_1$, then $uv\in E_r$ and so $vw\notin E(G)$, where $w$ is the vertex in $V(H)\less\{v_1, v_2\}$. But then we obtain a critical coloring of $G+vw$ from $\tau$ by first coloring the edge $vw$ red and then recoloring the edge $v_1v_2$ blue, a contradiction.  Thus   $u\ne v_1$. By the arbitrary choice of $u$, we may assume that  $v_1\notin V(H)$.   But then $vv_1\notin E(G)$ and   we obtain a critical coloring of $G+vv_1$ from $\tau$ by first coloring the edge $vv_1$ red and then recoloring the edge $v_1v_2$ blue, a contradiction.  This proves that every $(C_4,K_{1,2})$-co-critical graph on $n\ge5$ vertices has at least $2n-3$ edges.

 We next show that the bound in Theorem~\ref{C4P31} is sharp for all $n\ge 5$ by constructing a $(C_4,K_{1,2})$-co-critical graph with the desired number of edges. Let $n\ge 5$ be an integer and let $\varepsilon$ be the remainder of $n$ when divided by $2$.  Let $G_\varepsilon$ be the graph obtained from the cycle $C_{n-6+\varepsilon}$ by adding $6-\varepsilon$ new vertices $x_1,x_2,x_3,y_1,\ldots, y_{3-\varepsilon}$ and then joining: $x_3$ to all vertices in $V(C_{n-6+\varepsilon})\cup \{x_1,x_2,y_1,y_2\}$,   $x_1$ to all vertices in $\{x_2,y_1,y_3\}$ and  $x_2$ to all vertices in $\{y_2,y_3\}$ when   $\varepsilon=0$; $x_3$ to all vertices in $V(C_{n-6+\varepsilon})\cup \{x_1,x_2,y_1,y_2\}$ and $x_2$ to all vertices in $\{x_1,y_1,y_2\}$  when $\varepsilon=1$.  The construction of  $G_\varepsilon$  is  depicted in Figure~\ref{C4P3}.  It can be easily checked that  $e(G_\varepsilon) = 2n-3 $. It suffices to show that $G_\varepsilon$ is  $(C_4, K_{1,2})$-co-critical.

\begin{figure}[htbp]
\centering
\includegraphics[scale=0.8]{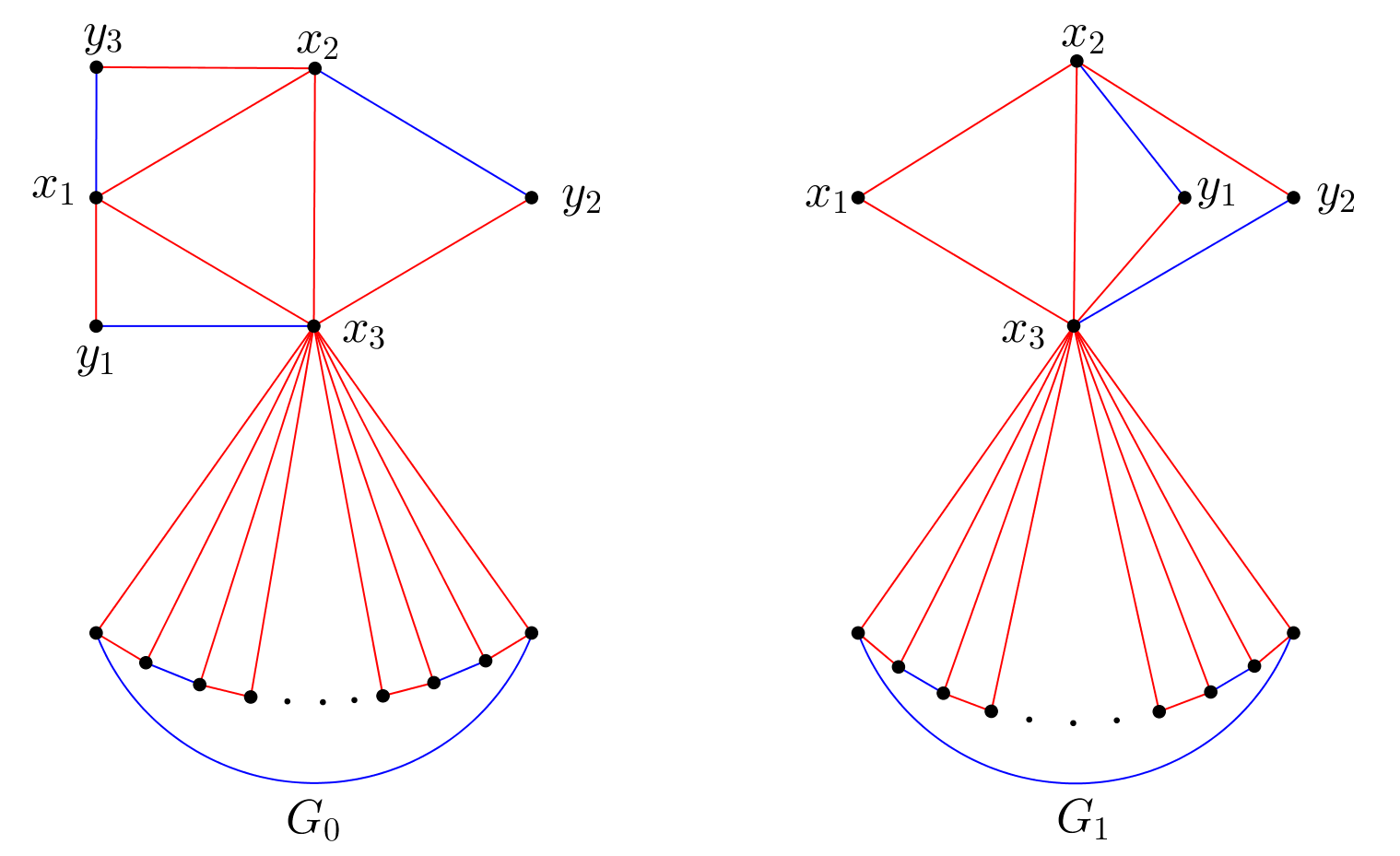}
\caption{Two $(C_4,K_{1,2})$-co-critical graphs with a unique critical coloring.
}
\label{C4P3}
\end{figure}

 Let  $\sigma: E(G_\varepsilon) \rightarrow \{$red, blue$\}$ be defined as  depicted in Figure~\ref{C4P3}.  It is simple to check that   $\sigma $   is a critical coloring  of $G_\varepsilon$. We next show that $\sigma$ is the unique critical coloring  of $G_\varepsilon$    up to symmetry.

  Let $\tau:E(G_\varepsilon) \rightarrow \{$red, blue$\}$ be an arbitrary critical coloring of   $G_\varepsilon$ with color classes $E_r$ and $E_b$.  It suffices to show that $\tau=\sigma$ up   to symmetry.  Let $G^\tau_r$ and $G^\tau_b$ be  $G_r$ and $G_b$ under the coloring $\tau$, respectively.   Then $G^{\tau}_r$ is $C_4$-free and $G^{\tau}_b$ is $K_{1,2}$-free. Thus $\Delta(G^{\tau}_b)\le 1$.
We first consider the case  when $\varepsilon=0$.  By Lemma~\ref{lem:blue}(ii),  $x_1x_2, x_1x_3, x_2x_3\in E_r$.  Since $G^{\tau}_r$ is $C_4$-free, we see that $d_b(y_i)=1$ for each $i\in[3]$. By symmetry, we may assume that $x_1y_3\in E_b$. Since $G^{\tau}_b$ is $K_{1,2}$-free, we see that $x_1y_1, x_2y_3\in E_r$, $x_2y_2, x_3y_1\in E_b$ and $x_3v\in E_r$ for each $v\in V(C_{n-6})\cup\{y_2\}$.  It then follows that $E_r\cap E(C_{n-6})$  and $E_b\cap E(C_{n-6})$ are two disjoint perfect matchings of $C_{n-6}$. Hence $\tau=\sigma$ up   to symmetry.  We next consider the case  when $\varepsilon=1$. By Lemma~\ref{lem:blue}(ii),  $x_2x_3\in E_r$.  Since $G^{\tau}_r$ is $C_4$-free, we see that $d_b(x_2)=d_b(x_3)=1$. By symmetry, we may assume that $x_2y_1, x_3y_2\in E_b$.  Since $G^{\tau}_b$ is $K_{1,2}$-free, we see that $x_2x_1, x_2y_2\in E_r$  and $x_3v\in E_r$ for each $v\in V(C_{n-5})\cup\{x_1, y_1\}$.  It then follows that $E_r\cap E(C_{n-5})$  and $E_b\cap E(C_{n-5})$ are two disjoint perfect matchings of $C_{n-5}$. Hence $\tau=\sigma$ up   to symmetry.

This completes the proof of Theorem~\ref{C4P31}.
\end{proof}

 \section*{Acknowledgement}
 The third author would like to thank John Schmitt for sending her reference~\cite{Ollmann}.

\end{document}